\documentclass[preprint]{elsarticle}
\usepackage{amssymb}
\usepackage{amsmath}
\usepackage{amsthm}
\usepackage{amsfonts}
\usepackage{mathrsfs}
\usepackage{hyperref}

\usepackage{latexsym}
\usepackage{tikz}

\usepackage{caption}
\usepackage{subcaption}
\usepackage{wrapfig}

\newtheorem{theorem}{Theorem}[section]
\newtheorem{lemma}[theorem]{Lemma}
\newtheorem{definition}[theorem]{Definition}
\newtheorem{proposition}[theorem]{Proposition}

\newtheorem{corollary}[theorem]{Corollary}
\newtheorem{example}[theorem]{Example}
\newtheorem{remark}[theorem]{Remark}

\def\mmset{{\mathcal B}}

\def\R{\mathbb{R}}
\def\bzero{\mathbf{0}}
\def\bunity{\mathbf{1}}

\def\conv{\operatorname{conv}_{\oplus}}
\def\spann{\operatorname{span}_{\oplus}}

\def\cB{\mmset}
\def\Lin{\operatorname{Lin}}
\def\supp{\operatorname{supp}}
\def\rank{\operatorname{rank}}

\begin{document}

\title{On the dimension of max-min convex sets}

\author[rvt1]{Viorel Nitica\fnref{fn1}}
\ead{vnitica@wcupa.edu}

\author[rvt2]{Serge{\u\i} Sergeev\corref{cor}\fnref{fn2}}
\ead{sergiej@gmail.com}

\address[rvt1]{University of West Chester, Department of Mathematics,
PA 19383, USA, and Institute of Mathematics, P.O. Box 1-764,
Bucharest, Romania}

\address[rvt2]{University of Birmingham, School of Mathematics, Edgbaston, B15 2TT, UK}

\cortext[cor]{Corresponding author.}

\fntext[fn1]{Supported by Simons Foundation Grant 208729.}

\fntext[fn2]{Supported by EPSRC grant EP/J00829X/1
and RFBR grant 12-01-00886.}

\begin{abstract} We introduce a notion of dimension of max-min convex sets, following the
approach of tropical convexity.
 We introduce a max-min analogue of the tropical rank of a matrix and show that it is equal to the dimension
 of the associated polytope. We describe the relation between this rank and the notion of strong regularity
 in max-min algebra, which is traditionally defined in terms of
unique solvability of linear systems and the trapezoidal property.
\end{abstract}


\maketitle

\section{Introduction\label{sec1}}

The max-min semiring is defined as the unit interval $\mmset=[0,1]$ with the operations
$a\oplus b:=\max(a,b)$, as addition, and $a\otimes b:=\min(a,b)$, as multiplication. The operations are idempotent,
$\max(a,a)=a=\min(a,a)$, and related to the
order:
\begin{equation}\label{first-eq-662}
\max(a,b)=b\Leftrightarrow a\leq b\Leftrightarrow \min(a,b)=a.
\end{equation}
One can naturally extend them to matrices and vectors
leading to the max-min (fuzzy) linear algebra
of~\cite{BCS-87,Cec-95, Gav-01,Gav:04, GP-07}. Note that in~\cite{GP-07}
the authors developed a more general version of max-min algebra over
arbitrary linearly ordered set, but we will not follow this
generalization here.

We denote by $\mmset(d,m)$ the set
of $d\times m$ matrices with entries in $\mmset$ and by $\mmset^d$
the set of $d$-dimensional vectors with entries in $\mmset$. Both
$\mmset(d,m)$ and $\mmset^d$ have a natural structure of semimodule
over the semiring $\mmset$.

A subset $V\subseteq\mmset^d$ is a {\em subsemimodule} if $u,v\in V$ imply
$u\oplus v\in V$ and $\lambda\otimes v\in V$ for all $\lambda\in\mmset$. Subsemimodules can be thought of
 as max-min analogue of subspaces or convex cones (especially in the
 context of the present paper).
In the max-min literature, subsemimodules arise as images of max-min
matrices or as eigenspaces. A subsemimodule $V\subseteq\mmset^d$
is said to be generated by a subset $X\subseteq\mmset^d$ and it is denoted by
$V=\spann(X)$, if it can be represented as a set of all {\em max-min linear combinations}
\begin{equation}
\label{mmlinear}
\bigoplus_{i=1}^m \lambda_i\otimes x^i\colon m\geq 1, \lambda_1,\dots,\lambda_m\in \mmset,
\end{equation}
of all $m$-tuples of elements $x^1,\ldots,x^m\in X$.

 The {\em max-min segment} between $x,y\in\mmset^d$ is defined as
\begin{equation}
\label{segm0}
\begin{aligned}
\ [x,y]_{\oplus} &= \{\alpha\otimes x\oplus \beta\otimes y\mid \,\alpha, \beta\in \mmset, \alpha \oplus \beta =1\}.\  
\end{aligned}
\end{equation}

A set $C\subseteq\mmset^d$ is called {\em max-min convex}, if it contains, with any two points
$x,y,$ the segment $[x,y]_{\oplus}$ between them. For a general subset $X\subseteq\mmset^d$,
define its {\em convex hull} $\conv(X)$ as the smallest max-min convex set containing $X$, i.e.,
the smallest set containing $X$ and stable under taking segments~\eqref{segm0}.
As in the ordinary convexity, $\conv(X)$ is the set of
all {\em max-min convex combinations}
\begin{equation}
\label{convX}
\bigoplus_{i=1}^m \lambda_i\otimes x^i\colon m\geq 1,\ \lambda_1,\dots,\lambda_m\in \mmset, \bigoplus_{i=1}^m \lambda_i=1,
\end{equation}
of all $m$-tuples of elements $x^1,\ldots,x^m\in X$.
The max-min convex hull
of a finite set of points is also called a {\em max-min convex polytope}.

The development of max-min convexity has been mostly inspired by new geometric techniques in max-plus (tropical)
linear algebra, like those developed in~\cite{AGG,DS-04,develin-etc, DLP}.
The development of tropical (max-plus) convexity
was started by K.~Zimmermann~\cite{Zim-77}, and it gained new
impetus after the works of Cohen, Gaubert, Quadrat and Singer~\cite{CGQS-05},
and  Develin, Sturmfels~\cite{DS-04}. This development has led
to many theoretical and algorithmic results, and in particular,
to new methods describing the solution set of max-plus
linear systems of equations~\cite{AGG,DLP}.

K.~Zimmermann~\cite{Zim-81} also suggested to develop
the convex geometry over wider classes of semirings
with idempotent addition, including the max-min semiring.
To the authors' knowledge, the case of
max-min semiring did not receive much interest in the past.
Though this case is mostly of
theoretical interest to us, it is also motivated by the
theory of fuzzy sets~\cite{HZim}. Some recent developments
in max-min convexity
include the description of max-min segments~\cite{NS-08I,Ser-03},
max-min semispaces~\cite{NS-08II}
and hyperplanes~\cite{Nit-10}, separation and non-separation results~\cite{NS-10,NS-11}. See~\cite{NS-13} for a
survey of max-min convexity that also includes some new results, in particular,
colorful extensions of the max-min Carath\'{e}odory theorem, as well as some
applications of the topological Radon theorem.

The present paper aims to develop a new geometric approach to the well-known notions of strong regularity and matrix rank in max-min algebra.  To this end, it seems to be the first paper that connects max-min linear algebra with max-min
convexity. Our main result is Theorem~\ref{t:mainres} stating that
the ``geometric'' dimension of a max-min polytope is equal to a max-min analogue of the tropical rank of the matrix whose columns are the ``vertices'' of that polytope.

 Let us make some preliminary observations. Note first that any subsemimodule is a max-min convex set. Moreover, since any max-min convex combination is just a max-min linear combination with one coefficient equal to $1$, we obtain that the max-min subsemimodules are precisely the max-min convex sets containing $0$. Thus for any $X\subseteq\mmset^d$, we have
\[
\spann(X)=\conv(X,\{0\}).
\]
We conclude that a finitely generated max-min semimodule can also be described as a max-min polytope with one ``vertex'' in the origin.

Conversely, if $C\subseteq\mmset^d$ is a max-min convex set, then
\[
V_C:=\{(\lambda\otimes x,\lambda)\mid x\in C,\,\lambda\in\mmset\}
\]
is a subsemimodule of $\mmset^{d+1}$. This construction is called {\em homogenization}.

The paper is organized as follows. The structure of
max-min segments is revisited in
Section~\ref{s:segments}. A notion of dimension in max-min
convexity is introduced and studied in Section~\ref{s:dim}.
Our approach is inspired by a geometric idea behind the notion the tropical rank~\cite{develin-etc}, that is,
a tropically convex polytope can be represented as a union of conventionally convex sets, and
its dimension can be defined as the greatest dimension of these convex sets.
In Section~\ref{s:rank} we introduce a notion of strong regularity and a notion of rank for a matrix $A$ over the max-min semiring. We
show that the rank of $A$, as we introduce it, is equal to the
dimension of the max-min convex hull of the columns of $A$.
In Section~\ref{s:sr} we show that our notion of strong regularity is
equivalent to the one traditionally studied in max-min algebra.
Thus it is closely related to the unique solvability of max-min
linear systems of the type $A\otimes x=b$ and, further, to the trapezoidal property of a matrix as studied, for example,
in~\cite{BCS-87,Cec-95,Gav-01, GP-07}.

\section{Max-min segments}
\label{s:segments}

In this section we describe general segments in $\mmset^{d},$
following~\cite{NS-08I,Ser-03}, where complete proofs can be found.
Note that the description of the segments in \cite{NS-08I, Ser-03}
is done for the equivalent case where $\mmset=[-\infty, +\infty]$.

Let $x=(x_{1},...,x_{d}),$ $y=(y_{1},...,y_{d})\in\mmset^{d},$
and assume that we are in the \emph{case of comparable endpoints},
say $x\leq y$ in the natural
order of $\mmset^{d}.$
Sorting the set of all coordinates $\{x_{i},y_{i},i=1,...,d\}$
we obtain a non-decreasing sequence, denoted by $t_1,t_2,\ldots, t_{2d}$.
This sequence divides the set $\mmset$ into
$2d+1$ subintervals
$\sigma_0=[0,t_{1}],\,\sigma_1=[t_1,t_2],...,\sigma_{2d}=[t_{2d},1]$,
with consecutive subintervals having one common endpoint.

Every point $z\in [x,y]_{\oplus}$ is represented as
$z=\alpha\otimes x\oplus\beta\otimes y$, where $\alpha=1$ or $\beta=1$.
However, case $\beta=1$ yields only $z=y$, so we can assume $\alpha=1$.
Thus $z$ can be regarded as a function of one parameter $\beta$, that is,
$z(\beta)=(z_{1}(\beta),...,z_{d}(\beta))$ with $\beta\in\mmset$.
Observe that for $\beta\in\sigma_0$ we have $z(\beta)=x$ and for
$\beta\in\sigma_{2d}$ we have $z(\beta)=y$. Vectors $z(\beta)$ with
$\beta$ in any other subinterval form a conventional {\em elementary segment}.
Let us proceed with a formal account of all this.

\begin{theorem}
\label{tcomp}
Let $x,y\in\mmset^d$ and $x\leq y$.
\begin{itemize}
\item[(i)] We have
\begin{equation}
\label{e:chain}
[x,y]_{\oplus}=\bigcup_{l=1}^{2d-1} \{z(\beta)\mid\beta\in\sigma_l\},
\end{equation}
where $z(\beta)=x\oplus(\beta\otimes y)$ and $\sigma_{\ell}=[t_l,t_{l+1}]$ for
$\ell=1,\ldots,2d-1$, and $t_1,\ldots,t_{2d}$ is the nondecreasing sequence
whose elements are the coordinates $x_i,y_i$ for $i=1,\ldots,d$.
\item[(ii)] For each $\beta\in\mmset$ and $i$, let
$M(\beta)=\{i\colon x_i\leq\beta\leq y_i\}$, $H(\beta)=\{i\mid\beta\geq y_i\}$
and $L(\beta)=\{i\colon\beta\leq x_i\}$. Then
\begin{equation}
\label{zibeta}
z_i(\beta)=
\begin{cases}
\beta, &\text{if $i\in M(\beta)$},\\
x_i, &\text{if $i\in L(\beta)$},\\
y_i, &\text{if $i\in H(\beta)$},
\end{cases}
\end{equation}
and $M(\beta),L(\beta), H(\beta)$ do not change in the interior
of each interval $\sigma_{\ell}$.
\item[(iii)]  The sets $\{z(\beta)\mid\beta\in\sigma_{\ell}\}$ in~\eqref{e:chain}
are conventional closed segments in $\mmset^d$ (possibly reduced to a point),
described by~\eqref{zibeta} where $\beta\in\sigma_{\ell}$.
\end{itemize}
\end{theorem}

For \emph{incomparable endpoints} $x\not\leq y,\,y\not\leq x,$
the description can be reduced to that of segments with comparable endpoints,
by means of the following observation.

\begin{theorem}
\label{tincomp}Let $x,y\in
\mmset^d$. Then $[x,y]_{\oplus}$ is the
concatenation of two segments with comparable endpoints, namely 
$\lbrack x,y]_{\oplus}=[x,x\oplus y]_{\oplus}\cup [y,x\oplus y]_{\oplus}.$
\end{theorem}

All types of segments for $d=2$ are shown in the right side of Figure 1.

The left side of Figure 1 shows, for the corresponding segments with comparable endpoints, a diagram, where for
$x=(x_1,x_2,x_3)$ and $y=(y_1, y_2,y_3)$, the intervals $[x_1,y_1], [x_2,y_2], [x_3,y_3],$ are placed
over one another, and their arrangement induces a tiling of the horizontal axis, which
shows the possible values of the parameter $\beta$.
The partitions of the intervals $[x_i,y_i], 1\le i\le 3,$ induced by this tiling are associated with the intervals
$\sigma_l$, and show the sets of {\em active indices} $i$ with $z_i(\beta)=\beta$.

\begin{figure}[h]
\begin{tikzpicture}[scale=.68]
\node at (4,3) {Segments in $\mmset^2$, comparable endpoints};
\draw [line width = 1] (0,0)--(2,0)--(2,2)--(0,2)--(0,0)--(2,2);
\draw [line width = 2] (.1, 1)--(.5,1)--(.5,1.7);
\draw [line width = 1] (3,0)--(5,0)--(5,2)--(3,2)--(3,0)--(5,2);
\draw [line width = 2] (3.1,.6)--(3.6,.6)--(4.5,1.5)--(4.5,1.9);
\draw [line width = 1] (6,0)--(8,0)--(8,2)--(6,2)--(6,0)--(8,2);
\draw [line width = 2] (6.1,.6)--(6.6,.6)--(7.5,1.5)--(7.9,1.5);
\draw [line width = 1] (0,-3)--(2,-3)--(2,-1)--(0,-1)--(0,-3)--(2,-1);
\draw [line width = 2] (.9,-2.9)--(.9,-2.5)--(1.7,-2.5);
\draw [line width = 1] (3,-3)--(5,-3)--(5,-1)--(3,-1)--(3,-3)--(5,-1);
\draw [line width = 2] (3.5,-2.9)--(3.5,-2.5)--(4.5,-1.5)--(4.5,-1.1);
\draw [line width = 1] (6,-3)--(8,-3)--(8,-1)--(6,-1)--(6,-3)--(8,-1);
\draw [line width = 2] (6.5,-2.9)--(6.5,-2.5)--(7.5,-1.5)--(7.8,-1.5);
\node at (4,-4) {Segment in $\mmset^2$, incomparable endpoints};
\draw [line width = 1] (3,-7)--(5,-7)--(5,-5)--(3,-5)--(3,-7)--(5,-5);
\draw [line width = 2] (3.4,-5.5)--(4.6,-5.5)--(4.6,-6.7);

\draw [line width = 1] (-10,-6)--(-2,-6);
\draw [line width = 2] (-9,-2)--(-5,-2);
\draw [line width = 2] (-8,-1)--(-3,-1);
\draw [line width = 2] (-7,0)--(-4,0);
\draw [line width = 1] [dotted] (-9,-6)--(-9,1);
\draw [line width = 1] [dotted] (-8,-6)--(-8,1);
\draw [line width = 1] [dotted] (-7,-6)--(-7,1);
\draw [line width = 1] [dotted] (-5,-6)--(-5,1);
\draw [line width = 1] [dotted] (-4,-6)--(-4,1);
\draw [line width = 1] [dotted] (-3,-6)--(-3,1);
\node at (-9,-6.5) {$t_1$};
\node at (-8,-6.5) {$t_2$};
\node at (-7,-6.5) {$t_3$};
\node at (-5,-6.5) {$t_4$};
\node at (-4,-6.5) {$t_5$};
\node at (-3,-6.5) {$t_6$};
\node at (-1.8,-6.5) {$\beta$};

\node at (-9.5,-5.5) {\tiny{$\emptyset$}};
\node at (-8.5,-5.5) {\tiny{$\{1\}$}};
\node at (-7.5,-5.5) {\tiny{$\{1,2\}$}};
\node at (-6,-5.5) {\tiny{$\{1,2,3\}$}};
\node at (-4.5,-5.5) {\tiny{$\{2,3\}$}};
\node at (-3.5,-5.5) {\tiny{$\{2\}$}};
\node at (-2.5,-5.5) {\tiny{$\emptyset$}};
\node at (-1.5,-5.5) {\tiny{$M(\beta)$}};

\node at (-9.5,-4.5) {$\sigma_0$};
\node at (-8.5,-4.5) {$\sigma_1$};
\node at (-7.5,-4.5) {$\sigma_2$};
\node at (-6,-4.5) {$\sigma_3$};
\node at (-4.5,-4.5) {$\sigma_4$};
\node at (-3.5,-4.5) {$\sigma_5$};
\node at (-2.5,-4.5) {$\sigma_6$};

\node at (-9.5,-2) {$x_1$};
\node at (-4.5,-2) {$y_1$};
\node at (-8.5,-1) {$x_2$};
\node at (-2.5,-1) {$y_2$};
\node at (-7.5,0) {$x_3$};
\node at (-3.5,0) {$y_3$};

\node at (-5.5,3) {Diagram showing intervals $\sigma_{\ell}$ and sets};
\node at (-5.5,2.3) {of coordinates moving together $M(\beta)$};
\end{tikzpicture}
\caption{Max-min segments\label{f:segments}.}
\end{figure}
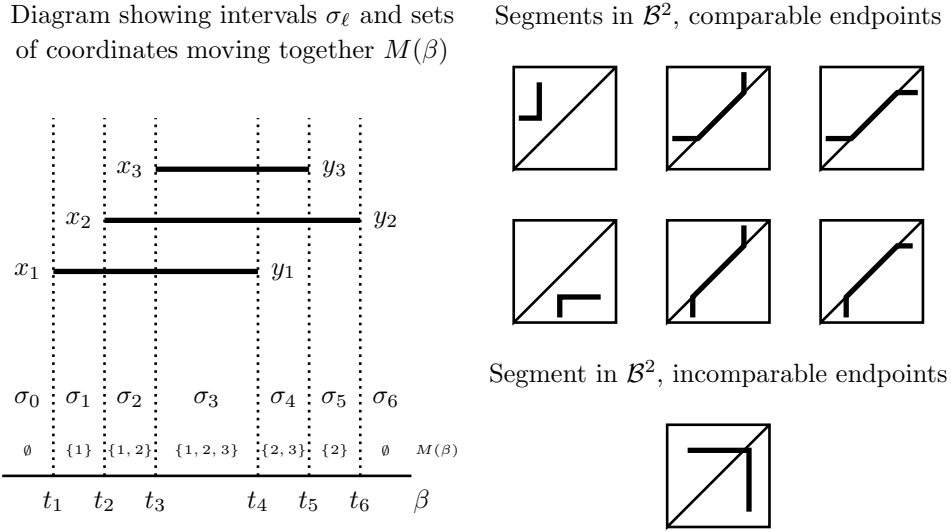
\medskip

\begin{remark}
\label{rsamerole}
{\rm It follows from the description above that each elementary segment is determined by a partition of the set of coordinates in two subsets. For points in the elementary segment, the coordinates in the first subset are constant and the coordinates in the second subset are all equal to a parameter running over a $1$-dimensional interval.
Therefore, similarly
to the max-plus case (see \cite{NS-1}, Remark 4.3) in $\mmset^d$ there are elementary segments in only $2^d-1$ directions. Elementary
segments are the "building blocks" for the max-min segments in $\mmset^d,$ in the sense that every segment $[x,y]\subseteq \mmset^d$ is the concatenation of a
finite number of elementary subsegments (at most) $2d-1$, respectively
$2d-2$, in the case of comparable, respectively incomparable, endpoints.
In the case of incomparable endpoints, the set of coordinates is partitioned
in two subsets of comparable coordinates, say of
cardinality $d_1,d_2,$ with $d_1+d_2=d$. The first subset
determines at most $2d_1-1$ elementary segments, and the second
set determines at most $2d_2-1$ elementary segments, for a total
of at most $2d-2$ elementary segments.}
\end{remark}

We close this section with an observation which we will need further. In this observation, as in the subsequent parts of the paper,
we will use the conventional arithmetic operations $(+,\cdot)$. For a real vector $y=(y_1,\dots,y_d)\in\R^d$, we define the support of $y$ 
 (with respect to the standard basis), as $\supp(y):=\{i\mid y_i\neq 0\}$.

\begin{lemma}
\label{l:mmsegm}
Let $y\in\mmset^d$ and let $u\in \R^d$ be a nonnegative real vector with support $\supp(u)=M$ such that $y+u\in\mmset^d$. Then
the following are equivalent:
\begin{itemize}
\item[{\rm (i)}] $[y,y+u]_{\oplus}$ contains only vectors $y+u'$ with $u'$ proportional to $u$;
\item[{\rm (ii)}] for all $i,j\in M$ we have $y_i=y_j$ and $u_i=u_j$.
\end{itemize}
\end{lemma}
\begin{proof}
(i)$\Rightarrow$(ii): By contradiction, let the condition of (ii) be violated. Suppose first that $y_i\neq y_j$ for some $i,j\in M$, and let
$M'\subseteq M$ be the proper subset of indices attaining $\min_{i\in M} y_i$. By Theorem~\ref{tcomp} (see also the left part of Figure~\ref{f:segments})
it follows that there is a nonnegative vector $u'$ such that $\supp(u')=M'$ and $y+u'$ belongs to the first subsegment of
$[y,y+u]_{\oplus}$. As $M'$ is a proper subset of $M$, it follows that
$u'$ is non-proportional to $u$.

Suppose now that $y_i=y_j$ for all $i,j\in M$ but $u_i\neq u_j$ for some $i,j\in M$. Let $M''\subseteq M$ be the
proper subset of indices attaining $\max_{i\in M} u_i$. By Theorem~\ref{tcomp} (see also the left part of Figure~\ref{f:segments})
it follows that there is a nonnegative vector $u''$ such that $\supp(u'')=M''$ and $y+u-u''$ belongs to the last subsegment of
$[y,y+u]_{\oplus}$. As $M''$ is a proper subset of $M$, it follows that
$u-u''$ is non-proportional to $u$.

(ii)$\Rightarrow$(i): By Theorem~\ref{tcomp}, in this case $[y,y+u]_{\oplus}$ is just the ordinary segment $\{y+u'\mid u'=\lambda u,\; 0\leq\lambda\leq 1\}$.
\end{proof}

\section{Dimension and max-min polytropes}
\label{s:dim}

The dimension of a max-min convex set can be introduced in the spirit of
the tropical rank, see for instance Develin, Santos,
Sturmfels~\cite[Section 4]{develin-etc}.
In this set-up we expect polytopes to be
representable as complexes of cells that are convex both in the usual and
in the new sense. We are interested in the interplay between these
convexities, similar to the case of tropical (max-plus) mathematics.

In what follows $\mmset^d$ has the usual Euclidean topology.
If $C\subseteq \mmset^d$, we denote
by $\overline{C}$ the closure of $C$ and by $\text{int}(C)$ the interior of $C$.

\begin{definition}
A max-min convex set $C\subseteq\mmset^d$, for $0\leq k\leq d$, is called a
$k$-dimensional open (resp. closed) max-min polytrope if it is also
a $k$-dimensional relatively open (resp. closed) conventionally
convex set.
\end{definition}

This concept is a max-min analogue of the so-called polytropes,
i.e., the sets which are (traditionally) convex and
tropically convex at the same time, see Joswig and Kulas~\cite{JK-09}.
Various types of
convex sets are shown in Figure 2.

\begin{figure}[h]
\begin{subfigure}{.45\textwidth}
\centering
\begin{tikzpicture}[scale=.5]
\draw [line width = 1] (0,0)--(10,0)--(10,10)--(0,10)--(0,0);
\draw [line width = 1][dotted] (0,0)--(10,10);
\draw [line width = 2] (3,8)--(6,8);
\draw [line width = 2] (9,3)--(9,6);
\draw [line width = 2] (6,6)--(7.8,7.8);
\draw [line width = 2] [fill=lightgray] (2,3)--(3,3)--(4,4)--(4,6)--(2,6)--(2,3);
\draw [fill=lightgray] (8,3) arc(0:-180:2);
\draw [line width = 2] (8,3) arc(0:-180:2);
\draw [line width = 2] (4,3)--(8,3);
\end{tikzpicture}
\caption{Closed max-min polytropes in
$\mmset^2$\label{f:polytropes}.}
\end{subfigure}
~~
\begin{subfigure}{.45\textwidth}
\centering
\begin{tikzpicture}[scale=.5]
\draw [line width = 1] (0,0)--(10,0)--(10,10)--(0,10)--(0,0);
\draw [line width = 1][dotted] (0,0)--(10,10);
\draw [line width = 2] [fill=lightgray] (1,5)--(4,5)--(1,8)--(1,5);
\draw [line width = 2] [fill=lightgray] (4,1)--(4,3.5)--(8,3.5)--(6,3.5)--(6,1)--(4,1);
\end{tikzpicture}
\caption{A max-min convex set that is not conventionally convex, and
hence not a max-min $2$-dimensional polytrope (below the diagonal),  and a
conventionally convex set that is not max-min convex in
$\mmset^2$\label{f:polytropes} (above the diagonal).}
\end{subfigure}
\caption{Max-min convex sets in
$\mmset^2$\label{f:polytropes}.}
\end{figure}

\begin{definition}
\label{def:dim} The dimension of a max-min convex set
$C\subseteq\mmset^d$, denoted by $\dim(C)$, is the greatest $k$ such
that $C$ contains a $k$-dimensional open polytrope.
\end{definition}

\begin{remark}
\label{r:useofdim}
{\rm Occasionally the notation $\dim$ will be also used for the usual dimension of
(conventionally) linear spaces and convex sets --- making sure that this will not
lead to any confusion.}
\end{remark}

Note that if the max-min convex set $C\subseteq \mmset^d$ has dimension $d$, then $C$  has nonempty interior.


In what follows we will make use of the usual linear algebra and
the usual convexity. For a convex set $C\subseteq\R^d$, let $C-y:=\{z-y\colon z\in C\}$,
and let $\Lin(C-y)$ be the least conventionally linear space containing $C-y$. From the convex analysis, recall that $C$ is relatively open if $C-y$ is open in $\Lin(C-y)$ for some, and hence for all $y\in C$. In this case, for any $u\in\Lin(C-y)$ there is $\epsilon>0$ such that $y+\epsilon u\in C$ and, conversely, if $y+u\in C$ then $u\in\Lin(C-y)$.

Observe that if $C$ is closed under componentwise maxima $\oplus$,
as in the case when it is a polytrope, then for each pair
$u,v\in\Lin(C-y)$ we have $y+\epsilon u, y+\epsilon v\in C$ for some
$\epsilon>0$ and $(y+\epsilon u)\oplus(y+\epsilon
v)=y+\epsilon(u\oplus v)\in C$, hence $u\oplus v\in\Lin(C-y)$. So
$\Lin(C-y)$ is also closed under taking componentwise maxima.
In particular, it follows that $\Lin(C-y)$ has a vector whose
support contains the support of any other vector in $\Lin(C-y)$,
that is, a vector whose support is the {\em largest} (by
inclusion).

The following auxiliary lemma, about the conventional linear algebra,
will be needed in the proof
of Theorem~\ref{t:quasibox}.

\begin{lemma}\label{l:aux53}  Let $L\subseteq \R^n$ be a linear subspace.
Assume that $L$ contains nonnegative vectors with largest support. Then:
\begin{equation}\label{e:neweq99221}
\Lin (L\cap \R^n_+)=L
\end{equation}
and, in particular, $\dim(\Lin (L\cap \R^n_+))=\dim(L)$.
\end{lemma}

\begin{proof} As $L\cap \R^n_+\subseteq L$, we always have  $\Lin (L\cap \R^n_+))=L$, so it suffices to
prove that $L$ can be generated by some vectors in $L\cap \R^n_+$, under the given condition.

Let $e\in L$ be a nonnegative vector with largest support and let
$\{f_1,\dots,f_k\}$ be a
basis for $L$. For every $i=1,\ldots,k$,
there exists $m_i>0$ such that $F:=\{f_1+m_1e,\dots,f_k+m_ke\}$ is a family of nonnegative vectors
in $L\cap \R^n_+$. The family $\tilde F:=F\cup \{e\}$ is a
family of nonnegative vectors in $L\cap \R^n_+$ that generates $L$ (since it generates all the base vectors),
so~\eqref{e:neweq99221} holds.
\end{proof}


The following result investigates
some of the interplay between the max-min and conventional convexities.
For a monograph in conventional
convexity see, e.g., Rockafellar~\cite{Rock}.

\begin{theorem}
\label{t:quasibox} Let $d\ge 1, 0\le k\le d$. Let $C\subseteq\mmset^d$ be a $k$-dimensional
open polytrope. Then for each point $y\in C$ there exist pairwise
disjoint index sets $J_1,\ldots, J_k\subseteq\{1,\ldots,d\}$ and
scalars $t_1,\ldots, t_k\in\mmset$ such that
\begin{itemize}
\item[(i)] $y_{\ell}=t_i$ for each $\ell\in J_i$ and $i\in\{1,\ldots,k\}$;
\item[(ii)] for some sufficiently small $\epsilon>0$, the set
\begin{equation}
\label{e:quasibox}
\begin{split}
B_y^{\epsilon}(J_1,\ldots,J_k):=&\times_{i=1}^k\{z^{J_i}\mid z_{\ell}^{J_i}=s_i,\, \forall\ell\in J_i,\;
t_i-\epsilon<s_i<t_i+\epsilon\}\times\\
&\times_{\ell\notin J}\{y_{\ell}\},
\end{split}
\end{equation}
where $J=J_1\cup\ldots\cup J_k$ and $z^{J_i}$ denotes a
(sub)vector with components indexed by $J_i$, is contained in $C$.
\end{itemize}
\end{theorem}
\begin{proof} Assume $C$ is not a point.
Given $y\in C$, consider $\Lin(C-y)$. We will show that it has a
nonnegative orthogonal basis. First, observe
that the max-min segments connecting $y$ with other points of $C$
give rise, possibly after a change of sign, to some nontrivial
nonnegative vectors in $\Lin(C-y)$. This follows from the description of
max-min segments given in Theorem~\ref{tcomp}.

Let us first show that the largest support of nonnegative vectors
in $\Lin(C-y)$ is equal to the largest support among all vectors of
$\Lin(C-y)$.
By contradiction, assume that the largest support of a
nonnegative vector is a proper subset $M\subset \{1,\ldots,d\}$,
achieved by a vector $u\in\Lin(C-y)$, and that there is a vector
$v\in\Lin(C-y)$ with some negative coordinates and support $\supp
v\not\subseteq M$. Possibly after inversion, $v$ has some positive
coordinates, whose indices do not belong to $M$. As $C$ is max-min
convex, we have $u\oplus v\in\Lin(C-y)$, and then $u\oplus v$ is a
nonnegative vector whose support strictly includes $M$, a
contradiction.

Thus we can assume that $\Lin(C-y)$ contains nonnegative vectors
with the largest support, hence by Lemma~\ref{l:aux53} the linear span of
its nonnegative part, the convex cone $K:=\Lin(C-y)\cap\R_+^d$, has
the same dimension $k$ as $\Lin(C-y)$. As $K$ is closed, by the
usual Minkowski theorem it can be represented as the set of positive
linear combinations of its extremal rays (recall that $w\in K$ is
called extremal if $u+v=w$ and $u,v\in K$ imply that $u$ and $v$ are
proportional with $w$), which generate the whole $\Lin(C-y)$. We
will prove that the extremal rays of $K$ have pairwise disjoint
supports.

By contradiction, let $u$ and $v$ be extremal rays of $K$,
not proportional with each other, with
$L:=\supp u\cap\supp v\neq\emptyset$.
We can assume that $\supp(u)=\supp(v)=L$ or that
$\supp u\neq\supp v$ and $(\supp
v)\backslash L\neq\emptyset$.
Take $\lambda>0$ and $\mu>0$ such
that $\lambda v_i> 2u_i$ and $u_i>\mu v_i$ for all $i\in L$. Hence we have
$(\lambda v-u)_i>\mu v_i$ for all $i\in L$. The
vector $w=\mu v \oplus (\lambda v-u)$ is nonnegative, below
$\lambda v$, and not proportional to $v$: in the case when $\supp(u)=\supp(v)=L$
 it is equal to $\lambda v-u$, and in the other case we have
$w_i=\mu v_i$ for $i\in (\supp v)\backslash L$ and $w_i>\mu v_i$ for
$i\in L$. We see that $w$ and $\lambda v-w$ are in $K$ not being
proportional to $v$, which contradicts that $v$ is extremal.

Thus we have proved that $\Lin(C-y)$ has an orthonormal basis
consisting of nonnegative vectors whose supports are pairwise
disjoint. The vectors of this basis (no more than $d$) also
generate the cone $K=\Lin(C-y)\cap\R_+^d$ being the extremals of
$K$. Now we use that $C$ is max-min convex and investigate the
properties of $y$ and the vectors of that basis. For a vector $u$
from the basis, there exists $\epsilon>0$ such that $y+\epsilon u$
belongs to $C$. From Lemma~\ref{l:mmsegm} we see that unless all components of $u$ are
equal to each to other and the corresponding components of $y$ are
equal to each other, we can find a vector $u'\leq u$ such that
$y+u'\in[y,y+\epsilon u]_{\oplus}\subseteq C$, where $u'$ is
non-proportional to $u$. Then $u=(u-u')+u'$ is not an extremal of $K$, a
contradiction.

So we obtained that for each $u$ in the nonnegative orthogonal basis
of $\Lin(C-y)$, all nonzero components of $u$ are equal to each
other, and the corresponding coordinates of $y$ are equal to each
other. Since the supports of the base vectors are pairwise disjoint,
this implies that $C$ contains a set of the form~\eqref{e:quasibox}.
More precisely, if the base vectors are denoted by $g^1,\ldots, g^k$
then we take $J_i=\supp(g^i)$ for $i=1,\ldots,k$. Since we can find $\epsilon$
such that $y+\epsilon g^i\in C$ for all $i$, we obtain that
$B_y^{\epsilon}(J_1,\ldots,J_k)\subseteq C$.
\end{proof}

\begin{definition}
\label{def:elem-quasibox} A set of the form~\eqref{e:quasibox} will
be called a ($k$-dimensional, open) {\rm quasibox.}
\end{definition}

\begin{lemma} A $k$-dimensional quasibox is a $k$-dimensional polytrope.
\end{lemma}

\begin{proof} A quasibox is obviously conventionally convex, so we
only need to show that it is max-min convex. Let $B:=B^{\epsilon}_y(J_1,\dots,J_k)$
be a quasibox defined by~\eqref{e:quasibox}, $z,\zeta\in  B$
and $\tau\in [z,\zeta]_{\oplus}$. Then
$\tau_{\ell}=y_{\ell}, \ell \not\in J$ and
$t_i-\epsilon<\tau_{\ell}<t_i+\epsilon$ if $\ell\in J_i$ due to the inequality
\[
\min(x,y)\le \max(\min(\alpha,x),\min(\beta,y))\le \max(x,y),
\]
which is true for all $x,y\in\mmset$ and
$\alpha,\beta\in \mmset$ such that $\max(\alpha,\beta)=1$, and which can be easily checked
by looking at all possible orders on $\{x,y,\alpha,\beta\}$.
\end{proof}

\begin{remark} {\rm As Figure~\ref{f:polytropes} shows, there are many polytropes that are
not quasiboxes.}
\end{remark}

\begin{corollary}
\label{c:dim-simple} The dimension $\dim(C)$ of a max-min convex set
$C\subseteq\mmset^d$ is equal to the greatest number $k$ such that $C$
contains a $k$-dimensional open quasibox.
\end{corollary}
\begin{proof}
Let $k=\dim(C)$. Then $C$ contains a $k$-dimensional (relatively)
open polytrope and, by Theorem~\ref{t:quasibox}, it also contains
a $k$-dimensional open quasibox. A quasibox of greater dimension
cannot be contained in $C$, since any quasibox is a polytrope.
\end{proof}

We now investigate the change of dimension under homogenization.
In fact, unlike in the usual convexity or max-plus convexity, the set $\lambda\otimes C:=\{\lambda\otimes x\mid x\in C\}$ does not look like a homothety of $C$, since the multiplication is not invertible. In particular, the dimension can also change. Consider the following example displayed on Figure~\ref{f:homothety}. Let $\lambda$ decrease from $1$ to $0$. Before $\lambda$ reaches $\lambda_4$ we have $\lambda\otimes C=C$. As $\lambda$ decreases from $\lambda_4$ to
$\lambda_3$, we see that $\lambda\otimes C$ is steadily ``swept'' towards the origin, but it still has a two-dimensional
region so that $\dim(\lambda\otimes C)=2$. The set $\lambda\otimes C$ becomes one-dimensional at $\lambda=\lambda_3$, consisting of two segments, one horizontal and one vertical. At $\lambda=\lambda_2$ the set $\lambda\otimes C$ becomes a single vertical segment, and at $\lambda=\lambda_1$ it shrinks to a point. The point moves towards the origin along the diagonal as $\lambda$ gets closer to $0$. The last subfigure displays the convex hull
$\conv(0,C)$, which is the least subsemimodule containing $C$, and also the projection of $V_C\subseteq \mmset^3$ onto the first $k=2$ coordinates.

\begin{figure}[h]
\begin{subfigure}{.4\textwidth}
\centering
\begin{tikzpicture}[scale=.5]
\draw (0,0)--(8,0)--(8,8)--(0,8)--(0,0)--(8,8);

\draw [fill=lightgray] (5,5)--(5,3)--(7,3)--(7,5)--(5,5);

\draw [line width = 2] (2, 5)--(7,5)--(7,1);
\draw [line width = 2] (5,5)--(5,3)--(7,3);

\node at (7,-.5) {$\lambda_4$};
\node at (5,-.5) {$\lambda_3$};
\node at (2,-.5) {$\lambda_2$};
\node at (1,-.5) {$\lambda_1$};

\draw [dotted, line width = 1] (7,0)--(7,3);
\draw [dotted, line width = 1] (5,0)--(5,5);
\draw [dotted, line width = 1] (2,0)--(2,5);
\draw [dotted, line width = 1] (0,1)--(7,1);
\draw [dotted, line width = 1] (1,1)--(1,0);
\end{tikzpicture}
\label{subf1}
\caption{Max-min convex set $C$}
\end{subfigure}
\qquad \qquad
\begin{subfigure}{.4\textwidth}
\centering
\begin{tikzpicture}[scale=.5]
\draw (0,0)--(8,0)--(8,8)--(0,8)--(0,0)--(8,8);

\draw [fill=lightgray] (5,5)--(5,3)--(7,3)--(7,5)--(5,5);

\draw [fill=darkgray] (5,5)--(5,3)--(6,3)--(6,5)--(5,5);

\draw [line width = 2] (2, 5)--(7,5)--(7,1);
\draw [line width = 2] (5,5)--(5,3)--(7,3);

\draw [line width = 2] (6,3)--(6,1);

\draw [line width = 2] (2,4)--(4,4)--(4,1);

\draw [line width = 2] (1.5,1)--(1.5,1.5);

\draw[fill] (.5,.5) circle [radius=.1];

\node at (7,-.5) {$\lambda_4$};
\node at (5,-.5) {$\lambda_3$};
\node at (2,-.5) {$\lambda_2$};
\node at (1,-.5) {$\lambda_1$};

\draw [dotted, line width = 1] (7,0)--(7,3);
\draw [dotted, line width = 1] (5,0)--(5,5);
\draw [dotted, line width = 1] (2,0)--(2,5);
\draw [dotted, line width = 1] (0,1)--(7,1);
\draw [dotted, line width = 1] (1,1)--(1,0);
\end{tikzpicture}
\label{subf2}
\caption{Max-min convex sets $\lambda C$}
\end{subfigure}
\quad\quad
\begin{subfigure}{.45\textwidth}
\centering
\begin{tikzpicture}[scale=.5]
\draw [fill=lightgray, line width = 2] (0,0)--(1,1)--(7,1)--(7,5)--(2,5)--(2,2)--(0,0);

\node at (7,-.5) {$\lambda_4$};
\node at (5,-.5) {$\lambda_3$};
\node at (2,-.5) {$\lambda_2$};
\node at (1,-.5) {$\lambda_1$};

\draw [dotted, line width = 1] (7,0)--(7,3);
\draw [dotted, line width = 1] (5,0)--(5,5);
\draw [dotted, line width = 1] (2,0)--(2,5);
\draw [dotted, line width = 1] (0,1)--(7,1);
\draw [dotted, line width = 1] (1,1)--(1,0);
\draw (0,0)--(8,0)--(8,8)--(0,8)--(0,0)--(8,8);
\end{tikzpicture}
\label{subf3}
\caption{Max-min convex hull of $C$ and $\bzero$}
\end{subfigure}
\caption{The behavior of $\lambda\otimes C$\label{f:homothety}}
\end{figure}
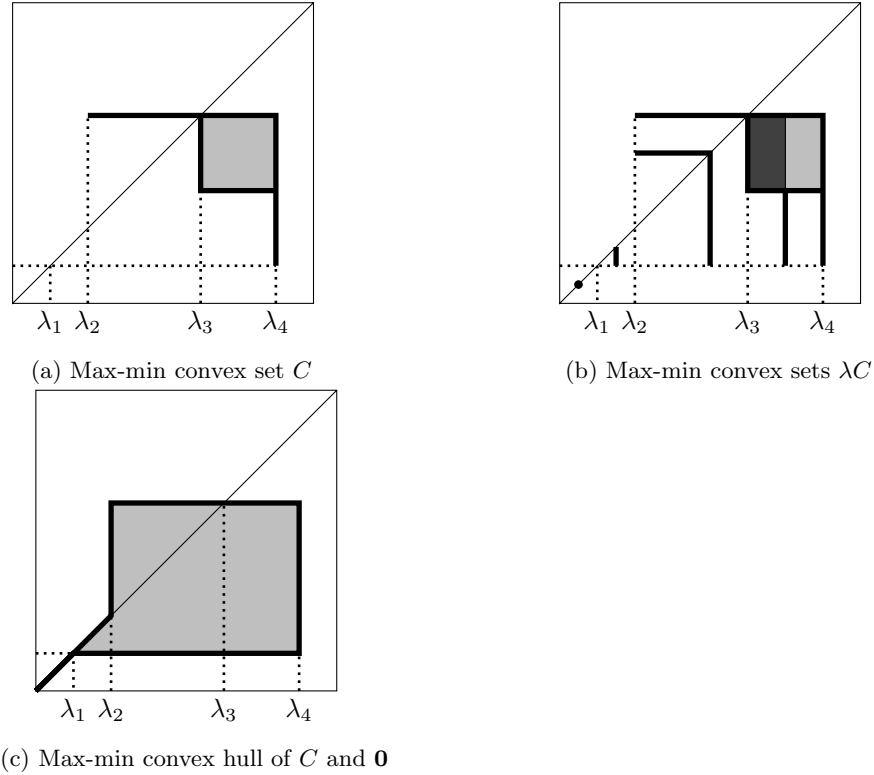

\begin{lemma}
\label{l:homothety}
Let $C\subseteq\mmset^d$ be a max-min convex set. Then
$\dim(\lambda\otimes C)\leq\dim(C)$ for all $0\leq\lambda\leq 1$.
\end{lemma}
\begin{proof}
Let $k=\dim(\lambda\otimes C)$. Then for some $y\in C$ that
satisfies condition (i) of Theorem~\ref{t:quasibox}, for some
numbers $t_1,\ldots,t_k$, and some subsets $J_1,\ldots,J_k$ of
$\{1,\ldots,d\}$, the set $\lambda\otimes C$ contains a quasibox
$B_y^{\epsilon}(J_1,\ldots,J_k)$ defined by~\eqref{e:quasibox}. As
$B_y^{\epsilon}(J_1,\ldots,J_k)\subseteq\lambda\otimes C$ we obtain
that $\lambda\geq t_i+\epsilon$ for all $i=1,\ldots,k$. Now let
$y=\lambda\otimes u$ for some $u\in C$ and consider any point $z\in
B_y^{\epsilon}(J_1,\ldots,J_k)$ with $z\geq y$.
{\bf Since $\lambda\geq t_i+\epsilon$ for all $i=1,\ldots,k$, we have
$u_j=y_j$ for all $j\in J_1\cup\ldots\cup J_k$, and hence
the components
$(u\oplus z)_j$ with $j\in J_1\cup\ldots\cup J_k$ are equal to those
of $y\oplus z=z$.} The components $(u\oplus z)_j$ with $j\notin
J_1\cup\ldots\cup J_k$ are equal to those of $u$, due to the fact
that in this case the components of $z$ coincide to the components
of $y$ and due to the formula $\max(a,\min(a,b))=a,$ which is true
for all $a,b\in \mmset$. So these components of $u$ are independent
of $z$. It follows that the points $u\oplus z$, for $z\geq y$ and
$z\in B_y^{\epsilon}(J_1,\ldots,J_k)$, form a set which contains
$B_x^{\epsilon/2}(J_1,\ldots,J_k)$, where $x_{\ell}=t_i+\epsilon/2$
for each $\ell\in J_i$ and $i\in\{1,\ldots,k\}$ and
$x_{\ell}=u_{\ell}$ for $\ell\notin J_1\cup\ldots\cup J_k$. However,
$z=\mu v$ for some $v\in C$ and hence $u\oplus z=u\oplus\lambda v\in
C$. It follows that $B_x^{\epsilon/2}(J_1,\ldots,J_k)\subseteq C$
and $\dim C\geq k$. The proof is complete.
\end{proof}

\begin{theorem}
\label{t:homothety}
Let $C\subseteq\mmset^d$ be a max-min convex set and let
$V_C\subseteq\mmset^{d+1}$ be the homogenization of $C$.
Then $\dim(V_C)=\dim(C)+1$.
\end{theorem}
\begin{proof}
We first prove that $\dim(V_C)\leq\dim(C)+1$. Suppose by
contradiction that $\dim(V_C)>\dim(C)+1$. Then $V_C$ contains a
polytrope of dimension at least $\dim(C)+2$.  For some $\mu$, the
section of $V_C$ by $\{u\in\mmset^{d+1}\mid u_{d+1}=\mu\}$ has a
nontrivial intersection with that polytrope, and that intersection
is a polytrope of dimension at least $\dim(C)+1$. But the section of
$V_C$ by $\{u\in\mmset^{d+1}\mid u_{d+1}=\mu\}$ is exactly
$(\mu\otimes C,\mu)$, and the dimension of $\mu\otimes C$ does not
exceed $\dim(C)$ by Lemma~\ref{l:homothety}. This contradiction
shows that $\dim(V_C)\leq\dim(C)+1$.

We now prove that $\dim(V_C)\geq\dim(C)+1$. For this, let
$k=\dim(C)$ and let $C$ contain a quasibox
$B_y^{\epsilon}(J_1,\ldots,J_k)$ defined by~\eqref{e:quasibox} as in
Theorem~\ref{t:quasibox}. Choosing a small enough $\epsilon$ we can
assume that $t_i+\epsilon<1$ for all
 $i=1,\ldots,k$. Let $J_{k+1}$ consist of the index $d+1$ and all indices of the
 components of $y$ that are equal to $1$. Choose $\epsilon$ such that $1-2\epsilon$ is greater than all $t_i+\epsilon$ and any coordinate of $y$ not equal to $1$,  and set $t_{k+1}:=1-\epsilon$. Define the components of $\tilde{y}\in\mmset^{d+1}$ by
 $\tilde{y}_{\ell}=t_i$ for each $\ell\in J_i$ and $i\in\{1,\ldots,k+1\}$, and $\tilde{y}_{\ell}=y_{\ell}$ otherwise.
 Then the homogenization of $B_y^{\epsilon}(J_1,\ldots,J_k)$, which is by definition the set
$$
\{(\mu\otimes x,\mu)\mid x\in B_y^{\epsilon}(J_1,\ldots, J_k),\;\mu\in\mmset\},
$$
 contains the quasibox  $B_{\tilde{y}}^{\epsilon}(J_1,\ldots,J_{k+1})$. As this homogenization is 
contained in $V_C$, the dimension of $V_C$ is at least $k+1$.
\end{proof}

\section{Dimension equals rank}
\label{s:rank}

In the remaining part of the paper, following the
parallel with the tropical rank considered by Develin, Santos
and Sturmfels~\cite{develin-etc} in the max-plus algebra,
we investigate how our notion of dimension relates with the notion of strong regularity in max-min algebra. For
$A\in\mmset(d,m+1)$, the $i$th column will be denoted by $A_{\bullet i}$.

\begin{definition}
\label{def:regular} A matrix $A\in\mmset(k,k+1)$ is called strongly regular
 if there exists an index $j: 1\leq j\leq k+1$, a bijection
 $\pi\colon\{1,\ldots,k\}\to\{1,\ldots,k+1\}\backslash\{j\}$ and
coefficients $\lambda_1,\ldots,\lambda_{j-1},\lambda_{j+1},\dots,\lambda_{k+1}\in\mmset$ such that in
the matrix
\begin{equation}
\label{lambda-matrix}
A[\lambda]:=(\lambda_1\otimes A_{\bullet 1},\dots, \lambda_{j-1}\otimes A_{\bullet j-1},
A_{\bullet j}, \lambda_{j+1}\otimes A_{\bullet j+1}, \dots, \lambda_{k+1}\otimes A_{\bullet k+1})
\end{equation}
the maximum in each row $i\in\{1,\ldots,k\}$ equals $\lambda_{\pi(i)}$ and is attained only by the term
$\pi(i)\in\{1,\ldots,k+1\}\backslash\{j\}$. We will say that the
coefficients $\lambda_i$ and bijection $\pi$ certify the strong regularity of $A$.
\end{definition}

\begin{remark}
\label{r:lambdas}
In Definition~\ref{def:regular}, the coefficients $\lambda_i$ are all nonzero. Furthermore, by slightly
decreasing these coefficients we can assume that they are all different and distinct from $1$ and the entries of $A$.
\end{remark}

For $A\in\mmset(d,m+1)$, let $\conv(A)$ denote the max-min convex hull
of the columns of $A$.

\begin{definition} Let $A\in\mmset(d,m+1)$. We call the max-min rank  and denote by $\rank(A)$ the largest integer
$k$ such that $A$ contains a strongly regular $k\times(k+1)$ submatrix.
\end{definition}

\begin{remark}
\label{r:rectang}
{\rm Note that the definition of strong regularity is introduced here for $k\times (k+1)$
rectangular matrices. A more usual ``square'' version of this
definition will appear in the next section, and we will show that it
is equivalent to the one studied in~\cite{BCS-87,Gav-01}.}
\end{remark}


The following theorem can be considered as one of the main result of this paper.

\begin{theorem}\label{t:mainres} Let $A=(a_{ij})\in \mmset(d, m+1)$.
Then $\dim(\conv(A))=\rank(A)$.
\end{theorem}

\begin{proof} We first suppose that $A$ contains a strongly regular
$k\times (k+1)$ submatrix, and show that $\dim(\conv(A))$ is at
least $k$. Without loss of generality we assume that this strongly
regular submatrix is extracted from the first $k$ rows and $k+1$
columns of $A$, and that $j=k+1$ in~\eqref{lambda-matrix}.
 Let $A'$ be the submatrix of $A$ extracted from the first $k+1$ columns.
 Since $\conv(A')\subseteq\conv(A)$, we have
 $\dim(\conv(A'))\leq\dim(\conv(A))$, so it suffices to prove that $\dim(\conv(A'))\geq k$.
For each column
$i:1\leq i\leq k$ there is a row where the maximum in
$A'[\lambda]$~\eqref{lambda-matrix} is attained only
by the $i$th term.  We assume that the
$\lambda_1,\ldots,\lambda_k$ are all different and distinct from the entries of $A'$.
With this, let $J_i$, for $1\leq i\leq k$, be the set of rows of $A'[\lambda]$ where
the only maximum is attained by the $i$th column and equals
$\lambda_i$. Let $J=J_1\cup\ldots\cup J_k$ and for $\ell\notin J$, if such indices exist,
let $\alpha_{\ell}$ be the maximum of the $\ell$th row of $A'[\lambda]$. Observe that this
maximum is equal to an entry of $A$. For each $i\colon 1\leq i\leq
k$, set
\begin{equation}
\label{e:mikappa}
\begin{split}
m_i&:=\max\{\max\{\alpha_\ell\colon \ell\notin J,\;\alpha_\ell<\lambda_i\},
\max\{\lambda_s\otimes a_{\ell s}\colon \ell\in J_i,\,s\neq i\}\},\\
\kappa&:= \min_{1\leq i\leq k} (\lambda_i-m_i).
\end{split}
\end{equation}
If the set $\{\alpha_{\ell}\colon\;\ell\notin J,\; \alpha_{\ell}
<\lambda_i\}$ is empty, then we assume that its maximum is zero. Observe that $\max\{\lambda_s\otimes a_{\ell s}\colon \ell\in
J_i,\,s\neq i\}< \lambda_i$ by the definition of $J_i$, hence
$m_i<\lambda_i$. For any vector
$\epsilon=(\epsilon_1,\ldots,\epsilon_k)$ such that $0\leq
\epsilon_i<\kappa$ for all $i=1,\ldots,k$, define the
vector-function $y(\epsilon)$:
\begin{equation}
\label{e:yepsdef}
y_{\ell}(\epsilon)=
\begin{cases}
\lambda_i-\epsilon_i, &\text{if $\ell\in J_i$ and $1\leq i\leq k$},\\
\alpha_{\ell}, &\text{if $\ell\notin J$}.
\end{cases}
\end{equation}
Using the definition of $\kappa$ and the fact that $y(0)$ is the
max-min linear combination of the columns of $A'$ with coefficients
$\lambda_1,\ldots,\lambda_k,1,$ we obtain that $y(\epsilon)$ is the
max-min linear combination  of the columns of $A'$ with coefficients
$\lambda_1-\epsilon_1,\ldots,\lambda_k-\epsilon_k,1$ for any
$\epsilon\colon 0\leq \epsilon_i<\kappa$ where $i=1,\ldots,k$.
Denote $\overline{y}:=y(\kappa/2,\ldots,\kappa/2)$. Then the
quasibox 
\[B_{\overline{y}}^{\kappa/2}(J_1,\ldots,J_k)=\{y(\epsilon): 0<\epsilon_i<\kappa,i=1,\ldots,k\},
\]
is contained in $\conv(A')$ and in $\conv(A)$. Since
$B_{\overline{y}}^{\kappa/2}(J_1,\ldots,J_k)$ is a $k$-dimensional
quasibox, this shows that $\dim(\conv(A))\geq k$.

Next, we have to show that given $k=\dim(\conv(A)$, there is a
strongly regular $k\times (k+1)$ submatrix. By
Theorem~\ref{t:quasibox}, $\conv(A)$ contains a $k$-dimensional
quasibox~\eqref{e:quasibox}. Then taking an element from each $J_i$,
consider the submatrix of $A$ extracted from the corresponding $k$
rows. Denote it by $A''\in\mmset(k,m+1)$. Assume that the rows are
$\{1,\ldots,k\}$. We will show that this submatrix contains an
strongly regular $k\times (k+1)$ submatrix. Indeed, being equal to
the projection of $\conv(A)$ onto the first $k$ coordinates
($(z_1,\ldots,z_k,\ldots,z_d)\mapsto(z_1,\ldots,z_k)$), $\conv(A'')$
contains the projection of the $k$-dimensional quasibox mentioned
above, and this is a usual $k$-dimensional box. This box contains a
point $x=(x_1,\ldots, x_k)$ whose all coordinates are different, and
distinct from the coefficients of $A''$. Since $x\in \conv(A'')$,
possibly permuting the columns of $A''$ we obtain that $(x_1,\ldots,
x_k)$ are the row maxima of the matrix
$$
(\mu_1\otimes A''_{\bullet 1},\dots, \mu_{m}\otimes A''_{\bullet m}, A''_{\bullet m+1}).
$$
Since $x_i$ are not equal to any entries of $A''$ and are all
different, we obtain that there is a $k$-element set
$N\subseteq\{1,\ldots,m\}$ and a bijection
$\pi\colon\{1,\ldots,k\}\to N$ such that $x_i=\mu_{\pi(i)}$ for all
$i=1,\ldots,k$, with all terms except for $\pi(i)$ being less than
$x_i$. This implies that the $k\times(k+1)$ submatrix extracted from
rows $1,\ldots,k$ and columns $\pi(1),\ldots,\pi(k),m+1$ is strongly
regular.

\end{proof}

\begin{definition} Let $A\in\mmset(d,m+1)$. We call \emph{the interior of $A$}
the interior of $\conv(A)$ in $\mmset^d$.
\end{definition}

\begin{corollary}
\label{c:interior} For $m\geq d$, $A\in\mmset(d,m+1)$ has nonempty
interior if and only if it contains a $d\times (d+1)$ strongly
regular submatrix.
\end{corollary}

\begin{proof} The corollary follows from Theorem~\ref{t:mainres} and Theorem~\ref{t:quasibox}.
\end{proof}

\section{Strong regularity: the link to max-min algebra}

\label{s:sr}

In this section we establish a close relation between our notion of
strong regularity and the one usually studied in max-min
algebra~\cite{BCS-87,BS-06,Cec-95,Gav-01,GP-07}. With this in mind,
let us define the notion of strong regularity for square matrices,
as a slight variation of Definition~\ref{def:regular}.

\begin{definition}
\label{def:srsquare} A matrix $A\in\mmset(k,k)$ is called strongly regular if there exists a bijection
 $\pi\colon\{1,\ldots,k\}\to\{1,\ldots,k\}$ and
coefficients $\lambda_1,\ldots,\lambda_{k}\in\mmset$ such that in
the matrix
\begin{equation}
\label{lambda-matrix-square}
A[\lambda]:=(\lambda_1\otimes A_{\bullet 1},\dots,\lambda_{k}\otimes A_{\bullet k})
\end{equation}
the maximum in each row $i\in\{1,\ldots,k\}$ equals $\lambda_{\pi(i)}$ and is attained only by the term
$\pi(i)\in\{1,\ldots,k\}$. We will say that the
coefficients $\lambda_i$ and bijection $\pi$ certify the strong regularity of $A$.
\end{definition}

\begin{remark}
\label{r:srsquare}
As in Definition~\ref{def:regular}, the coefficients $\lambda_1,\ldots,\lambda_k$ can be assumed to be different from
each other, distinct from the entries of $A$, $0$ and $1$.
\end{remark}

We will show later that this notion coincides with the one studied in max-min algebra.
The proof of the following statement is omitted.

\begin{lemma}
\label{l:srsquare} $A\in\mmset(k,k)$ is strongly regular in the sense of
Definition~\ref{def:srsquare} if and only
if $[A\;\bzero]$, where $\bzero$ is the $k$
component column of all zeros, is strongly regular in the sense of
Definition~\ref{def:regular}.
\end{lemma}

For $A\in\mmset(m,n)$ define $\hat{A}\in\mmset(m+1,n)$ by
\begin{equation}
\label{e:hata}
\hat{A}:=
\begin{pmatrix}
A\\
\bunity
\end{pmatrix},
\end{equation}
where $\bunity$ denotes the $n$-component row of all ones. Note that if $A\in\mmset(k,k+1)$ then
$\hat{A}\in\mmset(k+1,k+1)$ is square.

The mapping $A\to\hat{A}$ can be seen as a special case of homogenization. Indeed, if we set $C:=\conv(A)$, then we have
$V_C=\spann(\hat{A})$. This has the following immediate corollary.

\begin{corollary}
\label{c:ahata} $A\in\mmset(k,k+1)$ is strongly regular (in the sense of
Definition~\ref{def:regular}) if and only
if $\hat{A}\in\mmset(k+1,k+1)$ is strongly regular(in the sense of
Definition~\ref{def:srsquare}).
\end{corollary}
\begin{proof}
By Theorem~\ref{t:mainres}, $A\in\mmset(k,k+1)$ is strongly regular
if and only if $\dim(\conv(A))=k$, and $\hat{A}$ (that is,
$[\hat{A}\;\bzero]$) is strongly regular if and only if
$\dim(\conv([\hat{A}\;\bzero]))=\dim(\spann(\hat{A}))=k+1$.
Theorem~\ref{t:homothety} implies that these statements are
equivalent.
\end{proof}

\if{
\begin{proof}
``Only if'': If $A$ is strongly regular, then, assuming all
$\lambda_i$ certifying the strong regularity to be different from
$1$ and assuming that the certifying permutation $\pi$ is the
identity map from $\{1,\ldots,k\}$ to $\{1,\ldots,k\}$, we set
$\hat{\lambda_i}=\lambda_i$ for all $i=1,\ldots,k$ and
$\hat{\lambda}_{k+1}=1$. Coefficients $\hat{\lambda}_i$ for $i<k+1$
attain the unique maxima in the first rows of \eqref{lambda-matrix},
and $\hat{\lambda}_{k+1}$ attains the unique maximum in the last of
$\hat{A}$. Thus $\hat{\lambda}_i$ for $i=1,\ldots,k+1$ certify that
$\hat{A}$ is strongly regular

``If'': Let $\hat{A}$ be strongly regular, and assume that the
 identity mapping $\pi$ and the coefficients $\hat{\lambda}_i$ for $i=1,\ldots,k+1$ certify
that $\hat{A}$ is strongly regular In particular, the unique maximum
in the last row is attained by $\hat{\lambda}_{k+1}$. We show that
we can set $\hat{\lambda}_{k+1}:=1$. Indeed, observe that
$\hat{\lambda}_i>\hat{\lambda}_{k+1}\otimes  a_{i,k+1}$ and
$\hat{\lambda}_i<\hat{\lambda}_{k+1}$ for all $i<k+1$. Hence if we
set $\hat{\lambda}_{k+1}x=1$ then all other $\hat{\lambda}_i$ still
attain the unique maxima in the first $k$ rows of
$\hat{A}[\hat{\lambda}]$. So we can put $\lambda_i:=\hat{\lambda}_i$
for all $i=1,\ldots,k$, and these coefficients with the identity
permutation will certify that $A$ is strongly regular
\end{proof}
}\fi

\begin{definition}
\label{def:trap}
A matrix $A\in\mmset(m,n)$ is called {\rm trapezoidal} if the following condition holds:
\begin{equation}
\label{e:trap}
a_{ii}>\bigoplus_{\ell=1}^i \bigoplus_{t=\ell+1}^n a_{\ell t}\qquad\forall i=1,\ldots,m.
\end{equation}
\end{definition}

We now show that for $A\in\mmset(k,k)$ our notion of strong regularity is equivalent to the trapezoidal property, and hence it coincides
with the
strong regularity in max-min algebra introduced in~\cite{BCS-87}.

\begin{remark}
\label{r:BS} {\rm In fact, the equivalence between
Definition~\ref{def:srsquare} and Definition~\ref{def:trap} (for
square matrices) is known in max-min algebra. It follows, for
instance, from Butkovi\v{c} and Szabo~\cite[Theorem 2]{BS-06}.
However, we prefer to write the proofs of Proposition~\ref{p:srtrap}
and Theorem~\ref{t:sr} below for the sake of completeness and
convenience of the reader.}
\end{remark}

\begin{proposition}
\label{p:srtrap} $A\in\mmset(k,k+1)$ (or $A\in\mmset(k,k)$) is
strongly regular if and only if there exist permutation matrices $P$
and $Q$ such that $PAQ$ is trapezoidal.
\end{proposition}
\begin{proof}
We can assume that $A\in\mmset(k,k+1)$, since the other case is reduced to that case by
adjoining to $A\in\mmset(k,k)$ a zero column.

For the ``if'' part, we can assume that $A$ is trapezoidal. For
every row index $i$, we let $\lambda_i:=\alpha_i+\epsilon_i$, where
$\alpha_i$ equal the right-hand side of~\eqref{e:trap} and
$\epsilon_i$ are such that $\epsilon_1<\ldots<\epsilon_k$ and
$\alpha_i+\epsilon_i<a_{ii}$. Observe that
$\alpha_1\leq\ldots\leq\alpha_k$, and hence
$\lambda_1<\ldots<\lambda_k$. For $t>i$ we obtain $\lambda_i\otimes
a_{ii}>\lambda_t\otimes a_{it}$ since $a_{ii}>\lambda_i>a_{it}$ by
construction. For $t<i$ we obtain $\lambda_i\otimes
a_{ii}>\lambda_t\otimes a_{it}$ since $\lambda_i\otimes
a_{ii}=\lambda_i>\lambda_t$. Thus the coefficients
$\lambda_1,\ldots,\lambda_k$ and the identity permutation certify
that $A$ is strongly regular

The ``only if'' part: Let $A$ be strongly regular. Applying row and column permutations if necessary
(which corresponds to taking $PAQ$ as in the claim) we can assume that the
strong regularity is certified by the identity permutation
$\pi:\{1,\ldots,k\}\to\{1,\ldots,k\}$ and $\lambda_1,\ldots,\lambda_k$, which are
distinct from the entries of $A$ and satisfy
$0<\lambda_1<\lambda_2<\ldots<\lambda_k<1$.
In particular, we have $\lambda_i<a_{ii}$ for all $i$. For each $i$, we then have
$\lambda_i>a_{it}$ for all $t>i$, since $\lambda_i>\lambda_t\otimes a_{it}$ and $\lambda_i<\lambda_t$. Hence we also have
$a_{ii}>\lambda_i>\lambda_{\ell}>a_{\ell t}$ for all $\ell<i$ and $\ell<t$. Thus the trapezoidal property follows.
\end{proof}

Recall that by Corollary~\ref{c:ahata}, $A\in\mmset(k,k+1)$ is
strongly regular if and only if $\hat{A}$ is strongly regular In
fact, this is also easy to see by means of the trapezoidal property.
We now conclude with the following observation, which is similar
to~\cite[Theorem 3]{BCS-87}.

\begin{theorem}
\label{t:sr}
Let $A\in\mmset(d,k+1)$ with $d\geq k$. Then the following are equivalent:
\begin{itemize}
\item[{\rm (i)}] there exists a vector $b\in\mmset^d$ such that the system $A\otimes x=b$ has a positive solution, which is also the unique solution that satisfies $\bigoplus_{i=1}^{k+1} x_i=1$;
\item[{\rm (ii)}] there exists a vector $\hat{b}\in\mmset^{d+1}$ such that the system $\hat{A}\otimes x=\hat{b}$ has a
    positive solution, which is also the unique solution to that system;
\item[{\rm (iii)}] $\hat{A}$ contains a $(k+1)\times(k+1)$ strongly regular submatrix;
\item[{\rm (iv)}] $A$ contains a $k\times (k+1)$ strongly regular submatrix.
\end{itemize}
\end{theorem}
\begin{proof}
(i)$\Rightarrow$(ii): Take $\hat{b}=(b\; 1)^T$. Observe that
$\bigoplus_{i=1}^{k+1} x_i=1$ is satisfied for any solution of
$\hat{A}\otimes x=\hat{b}$, and then (i) shows that it is unique.\\
(ii)$\Rightarrow$(iii):
For this we can exploit, e.g.,\cite[Theorerm 3]{BCS-87}.\\
(iii)$\Leftrightarrow$(iv): Equivalence between these statements follows from Theorem~\ref{t:mainres} and Theorem~\ref{t:homothety}.
Alternatively, we can use the existence of permutation matrices $P$ and $Q$ such that $A$ or $\hat{A}$ have a trapezoidal submatrix (Proposition~\ref{p:srtrap}).\\

\if{
Applying the transformation
$\hat{A}\to P\hat{A}Q$ where $P$ and $Q$ are permutation matrices,
we can obtain a matrix such that the $(k+1)\times (k+1)$ submatrix extracted from its first $k+1$ rows is trapezoidal.
If this submatrix contains the row of all ones, then
this is the last row.  In any case, the first $k$ rows constitute a
trapezoidal $k\times (k+1)$ submatrix, hence (iv).\\
}\fi \if{ If the strongly regular submatrix of $\hat{A}$ contains
the last row, then the implication follows from Prop.~\ref{p:ahata}.
If it does not contain the last row, then we can assume that the
submatrix extracted from the first $k+1$ rows is strongly regular
and the coefficient $\lambda_i$ attains the unique maximum in the
$i$th row, for $i=1,\ldots,k$. Assuming that all $\lambda$ are
different, we identify the greatest of them. The row in which this
coefficient attains the unique maximum is then replaced by the last
row of $\hat{A}$ giving rise to a strongly regular submatrix
containing the last row. }\fi

(iv)$\Rightarrow$(i): Let
$\lambda_1,\ldots,\lambda_k$ and the identity permutation certify
the strong regularity of the $k\times(k+1)$ submatrix extracted from
the first $k$ rows of $A$. Assume that the values of
$\lambda_1,\ldots,\lambda_k$ are all different and distinct from the
entries of $A,$ as well as $0$ and $1$. Define the components of $b$
to be the maxima in the rows of $A[\lambda]$, then the first $k$
components of $b$ are equal to $\lambda_1,\ldots,\lambda_k$.

Let $A'$ be the strongly regular $k\times (k+1)$ submatrix extracted
from the first $k$ rows of $A$. The corresponding subvector of $b$
is $b'=(\lambda_1,\ldots,\lambda_k)$, and
$x=(\lambda_1,\ldots,\lambda_k,1)$ is a solution to $A'\otimes
x=b'$ and $A\otimes x=b$. As the entries of $b'$ are all different from the entries of
$A'$, any other solution $y$ with a component equal to $1$ contains
all these components $\lambda_1,\ldots,\lambda_k,1$, possibly
permuted. However, then $x\oplus y$ is also a solution where some of
the components $\lambda_i$ are lost, since we chose them to be all
different. This is a contradiction, which shows that $A'\otimes
x=b'$ is uniquely solvable with $(\lambda_1,\ldots,\lambda_k,1)$
(requiring one $1$ component), which implies the same for $A\otimes
x=b$.
\end{proof}

In particular, $A\in\mmset(d,k+1)$ contains a
strongly regular $k\times (k+1)$ submatrix if and only if there
exist permutation matrices $P$ and $Q$ such that $P\hat{A}Q$
contains a trapezoidal $(k+1)\times(k+1)$ submatrix. To find such a
submatrix, that is, to verify that the equivalent conditions of
Theorem~\ref{t:sr} hold, we can apply the strongly polynomial
algorithm of~\cite{BCS-87}.

We conclude with two sufficient conditions for a matrix to have low rank.

\begin{proposition}
\label{p:critera2} If $A\in\mmset(d,m+1)$ is such that
for any $\lambda_1,\ldots,\lambda_{m+1}\in\mmset$ with $\lambda_j=1$
for some $j\in\{1,\ldots,m+1\}$ there exist $k$ columns such that
the maximum in every row of $A[\lambda]$~\eqref{lambda-matrix}  is
attained in one of these $k$ columns, then $\dim(\conv(A))\leq k$.
\end{proposition}
\begin{proof}
Observe that there is no regular $s\times (s+1)$ submatrix with $s>k$, and apply Theorem~\ref{t:mainres}.
\end{proof}

\begin{corollary}[Sufficient condition for $\dim(\conv(A))\leq 2$]\label{c:emptint}
Let $A=(a_{ij})\in \mmset(d,m)$  satisfy
\begin{equation}\label{condition11}
\max_{1\le k\le d} a_{ki}\le \min_{1\le k\le n} a_{k,i+1},\quad\forall i\colon 1\le i\le m
\end{equation}
Then $\dim(\conv(A))\leq 2$.
\end{corollary}
\begin{proof}
Let $x$ be a max-min convex combination of the columns of $A$,
  with coefficients $\lambda_1,\ldots,\lambda_m$ such that $\lambda_j=1$ for some $j\in\{1,\ldots,m\}$.  Consider the matrix $A[\lambda]$~\eqref{lambda-matrix}. We will show that there are two columns
where all row maxima of~\eqref{lambda-matrix}
are attained. For this, let $I$ be the set
of column indices $i$ where $\lambda_i>\min_k a_{ki}$, and let $J$ be the complement
of this set.

Considering the submatrix of $A[\lambda]$~\eqref{lambda-matrix}
extracted from the columns in $I$ we see that all row maxima are attained in the
column with the biggest index.  All coefficients of a column in $J$ are equal to
each other. Therefore, in the submatrix of~$A[\lambda]$
extracted from the columns in $J$ there is also a column where all row maxima
are attained. This column and the column with biggest index in $I$ are the
two columns where all row maxima of~$A[\lambda]$
are attained (possibly, there may be other such columns, but they are redundant).
By Proposition~\ref{p:critera2} this shows that $\dim(\conv(A))\leq 2$.
\end{proof}

\if{ In particular, it is not true that if all entries of a matrix
$A\in\cB^{n\times (n+1)}$ are different, then the max-min convex
hull of the columns of $A$ has a non-empty interior.}\fi

\begin{example} The max-min polytope
$\conv(A) \subseteq \mmset^3$ generated by the matrix
\begin{equation}
A=
\begin{pmatrix}
.01 & .02 & .03& .04\\
.05 & .06 & .07 & .08\\
.09 & .10 & .11 & .12
\end{pmatrix}
\end{equation}
has non-empty interior, meaning that $\dim(\conv(A))=3$. To see that $A$ is strongly regular, choose
$j=1$ and $\lambda_2=.10, \lambda_3=.07, \lambda_4=.04$. A trapezoidal form of $A$ can be obtained
by reversing the order of columns:
\begin{equation}
A=
\begin{pmatrix}
.04 & .03 & .02& .01\\
.08 & .07 & .06 & .05\\
.12 & .11 & .10 & .09
\end{pmatrix}
\end{equation}

\end{example}

\begin{example} The max-min polytope $\conv(A) \subseteq \mmset^3$ generated by the matrix
\begin{equation}
A=
\begin{pmatrix}
.01 & .04 & .07& 10\\
.02 & .05 & .08 & .11\\
.03 & .06 & .09 & .12
\end{pmatrix}
\end{equation}
has $\dim(\conv(A))=2$. The inequality $\dim(\conv(A))\leq 2$
follows from Corollary~\ref{c:emptint}, as condition
\eqref{condition11} is satisfied for all $i\colon 1\le i\le 3$. A
regular $2\times 3$ submatrix can be extracted from rows $1$ and $3$ and
columns $1,3,4$: set $j=1$, $\lambda_3=.09$ and $\lambda_4=.08$.
\end{example}

\section*{Acknowledgement}
The authors are grateful to the anonymous referees, whose comments have contributed to the
clarity of the proofs in this paper.

\end{document}